\documentclass[a4paper, 12pt]{amsart}
\usepackage[a4paper,hmargin=2.5cm,vmargin=2.5cm]{geometry}
\usepackage[T2A]{fontenc}
\usepackage[utf8]{inputenc}

\usepackage[russian,english]{babel}
\usepackage{amsmath,amssymb,amsthm,amsfonts,mathrsfs,amscd}
\usepackage[linkcolor=blue,citecolor=blue,colorlinks=true,urlcolor=blue]{hyperref}
\usepackage{comment}
\usepackage{graphics}

\numberwithin{equation}{section}
\newtheorem{theorem}{Theorem}[section]
\newtheorem{lemma}[theorem]{Lemma}
\newtheorem{propos}[theorem]{Proposition}
\theoremstyle{definition}
\newtheorem{definition}[theorem]{Definition}

\newtheorem{proof}{Proof}
\let\origendproof\endproof
\def\endproof{\unskip\nobreak\hskip5pt plus 1fill$\square$\origendproof}

\newtheorem{sled}[theorem]{Corollary}
\newtheorem{ex}[theorem]{Example}
\newtheorem{rem}[theorem]{Remark}

\def\Gal{\mathop{\rm Gal}\nolimits}

\def\PGL{\mathop{\rm PGL}\nolimits}
\def\GL{\mathop{\rm GL}\nolimits}

\def\Aut{\mathop{\rm Aut}\nolimits}
\def\rk{\mathop{\rm rk}\nolimits}
\def\tr{\mathop{\rm tr}\nolimits}
\def\det{\mathop{\rm det}\nolimits}
\def\Pic{\mathop{\rm Pic}\nolimits}

\begin{document}
\title{Automorphisms of two-dimensional quadrics}
\author{A.\,V.~Zaitsev}
\address{National research university ''Higher school of economics'', Laboratory of algebraic geometry, 6 Usacheva str., Moscow, 119048, Russia}
\email{\href{alvlzaitsev1@gmail.com}{alvlzaitsev1@gmail.com}}
\thanks{The study has been funded within the framework of the HSE University Basic Research Program.} 

\maketitle
\medskip


\begin{abstract}
In this paper, we find the maximum values that the Jordan constant of the automorphism group of a smooth two-dimensional rational quadric over a field of characteristic zero can attain, depending on the arithmetic properties of a field.
\end{abstract}

\tableofcontents

\section{Introduction}
It is often useful to study infinite groups at the level of their finite subgroups. For example, one can study the Jordan property of infinite groups.

\begin{definition}[{\cite[Definition 2.1]{Pop}}]\label{def: J.} Let $G$ be a finite group. The Jordan constant $J(G)$ of $G$ is the smallest index of a normal abelian subgroup in $G$. Let $\Gamma$ be an arbitrary group. Then $\Gamma$ is called \textit{Jordan} if the value $$ J(\Gamma) = \sup\limits_{G\subseteq \Gamma,\ |G|<\infty }(J(G))$$ is finite. In this case the number $J(\Gamma)$ is called the Jordan constant of the group $\Gamma$.
\end{definition}
Important examples of infinite Jordan groups are the complete linear groups $\text{GL}_n(K)$ over a field $K$ of characteristic zero. The fact that these groups are Jordan was proved by Camille Jordan, see~\cite[\S 40]{Jord} or~\cite[Theorem 36.13]{CR}, and their Jordan constants over algebraically closed fields were computed in~\cite{Coll}. As a corollary, all linear algebraic groups are Jordan groups. In particular, the projective linear groups $\PGL_n(K)$, which are automorphism groups of projective spaces.

The next natural question is about the group of birational automorphisms of the projective plane --- is this group Jordan or not? The fact that this group is Jordan over fields of characteristic zero was proved in the paper~\cite[Theor\'em\`e 3.1]{Serre}. The situation with it's Jordan constants is more complicated for these groups than for linear groups. At the moment, the exact values of the Jordan constants of the group of birational automorphisms of the projective plane have been computed over algebraically closed fields of characteristic zero, over fields of real and rational numbers, see~\cite{Yas}.

In this paper, we deal with one of the most important steps of a question about Jordan constants of the group of birational automorphisms of the projective plane over fields of characteristic zero. Namely, finite subgroups in the group of birational automorphisms of the projective plane act effectively on rational del Pezzo surfaces or on rational surfaces with a conic bundle structure. Therefore, it is useful to understand the Jordan constants of the automorphism groups of these surfaces. The results about the rational del Pezzo surface of degree 9, that is, about the projective plane, are obtained in~\cite{Hu}. In this paper we compute the Jordan constants of automorphism groups of rational del Pezzo surfaces of degree $8$, that is, smooth two-dimensional rational quadrics. Recall that a smooth two-dimensional quadric is rational if and only if it contains a rational point.

We are interested in the following value: $$ M(K) = \max\limits_{X}(J(\Aut(X)),$$
where the maximum is taken over smooth rational quadrics in $\mathbb{P}^3_K$. As a result, we prove the following theorem.

\begin{theorem}\label{theo:M(K)}
Let $K$ be a field of characteristic $0$.
\begin{enumerate}
\item $M(K) = 7200$ if and only if $\sqrt{5} \in K$, and $-1$ is a sum of two squares in $K$;
\item $M(K) = 120$ if and only if $\sqrt{5} \notin K$, and $-1$ is a sum of two squares in $K(\sqrt{5})$;
\item $M(K) = 60$ if and only if $\sqrt{5} \in K$, and $-1$ is not a sum of two squares in $K$; 
\item $M(K) = 8$ if and only if $\sqrt{5} \notin K$, and $-1$ is not a sum of two squares in $K(\sqrt{5})$.
\end{enumerate}

\end{theorem}

\begin{sled}\label{sled: M(Q), M(R), M(C)...} All values of $M(K)$ from Theorem~\ref{theo:M(K)} are attained:
$$M(\mathbb{Q}) = 8, \; M(\mathbb{R}) = 60,\;   M(\mathbb{Q}(i)) = 120, \; M(\mathbb{Q}(\sqrt{-7})) = 120,\; M(\mathbb{C}) = 7200.$$
\end{sled}

During the proof of Theorem~\ref{theo:M(K)}, we compute the Jordan constant of a surface $\mathbb{P}^1_K \times \mathbb{P}^1_K$, more precisely, we prove the following theorem.


\begin{theorem}\label{theo:theo P^1 x P^1}
Let $K$ be a field of characteristic $0$.
\begin{enumerate}
    \item $J\left(\Aut\left(\mathbb{P}^1_K \times \mathbb{P}^1_K\right)\right) = 7200$ if and only if \mbox{$\sqrt{5} \in K$} and $-1$ is a sum of two squares in $K$;
    \item $J(\Aut(\mathbb{P}^1_K \times \mathbb{P}^1_K)) = 72$ if and only if \mbox{$\sqrt{5} \not\in K$}, and $-1$ is a sum of two squares in~$K$;
    \item $J(\Aut(\mathbb{P}^1_K \times \mathbb{P}^1_K)) = 8$ if and only if $-1$ is not a sum of two squares in $K$.

\end{enumerate}
\end{theorem}

\begin{ex}\label{C R Q}
All values of Jordan constant from Theorem~\ref{theo:theo P^1 x P^1} are attained:
\begin{itemize} 
    \item $J(\Aut(\mathbb{P}^1_{\mathbb{C}} \times \mathbb{P}^1_{\mathbb{C}})) = 7200$, \item $J(\Aut(\mathbb{P}^1_{\mathbb{Q}(i)} \times \mathbb{P}^1_{\mathbb{Q}(i)})) = 72$, 
    \item $J(\Aut(\mathbb{P}^1_\mathbb{R} \times \mathbb{P}^1_\mathbb{R})) = J(\Aut(\mathbb{P}^1_\mathbb{Q} \times \mathbb{P}^1_\mathbb{Q})) = 8$.
\end{itemize}
\end{ex}

Also we prove the following useful proposition.

\begin{propos}\label{prop: J(AutS) <= 120 if rkPic = 1}
    Let $K$ be a field of characteristic $0$. Let $S$ be a smooth rational quadric in~$\mathbb{P}^3_K$ and $S \not\simeq \mathbb{P}^1_K \times \mathbb{P}^1_K$. Then $$J(\Aut(S)) \leqslant 120.$$
\end{propos}



The plan of the paper is as follows. In Section~\ref{sect:prelim} we collect some auxiliary statements from group theory. In Section~\ref{sect: semidirect 1}, we compute the Jordan constants of groups of the form~$(\Gamma \times \Gamma) \rtimes \mathbb{Z}/2\mathbb{Z}$, which are similar to the automorphism group of the surface $\mathbb{P}^1\times\mathbb{P}^1$. In Section~\ref{sect:P1xP1} we prove Theorem~\ref{theo:theo P^1 x P^1}. In Section~\ref{sect: semidirect 2} we find the matrices generating the group~$\mathfrak{A}_5$ inside $\PGL_2(L)$, and using them we estimate the Jordan constants of automorphism groups of smooth rational two-dimensional quadrics different from $\mathbb{P}^1\times\mathbb{P}^1$. Finally, in Section~\ref{sect: main theorem} we prove Theorem~\ref{theo:M(K)}, Proposition~\ref{prop: J(AutS) <= 120 if rkPic = 1} and Corollary~\ref{sled: M(Q), M(R), M(C)...}.

We will use the following notation. We denote the neutral element of a group by $e$. We denote the dihedral group of order~${2n}$ by~$D_{2n}$. We denote the algebraic closure of the field~$K$ by $\overline{K}$. If $K \subset L$ is an extension of fields, and $X$ is a variety over $K$, then we denote the extension of scalars of $X$ to $L$ by $X_L$. If $H\subset G$ are groups and $g, g'\in G$, then we denote a subgroup of $G$ generated by all elements of the subgroup $H$ and the element~$g$ by $\langle H, g\rangle$, and we denote a subgroup generated by elements~$g$ and~$g'$ by $\langle g,g'\rangle$. We denote the group defined by the set of generators $S$ and the list of relations $R$ by~$\langle S\mid R\rangle$.

\textbf{Acknowledgements.} I would like to thank my advisor Constantin Shramov for stating the problem, useful discussions and constant attention to this work. I also want to thank Andrey Trepalin for useful discussions and especially for elegant completion of proof of Lemma~\ref{lemma: J = 60 for every quadratic extension}. The work was supported by the Theoretical Physics and Mathematics Advancement Foundation ``BASIS''.

\section{Jordan constants and group theory}\label{sect:prelim}

In this section, we collect some auxiliary statements from group theory. The following lemma is obvious and will be used without reference to it.
\begin{lemma}\label{lemma:J_1<J_2}
Let $H$ be a subgroup of a Jordan group $G$. Then $H$ is also Jordan and~\mbox{$J(H)\leqslant J(G)$}.
\end{lemma}

The following lemma is also standard and simple.

\begin{lemma}[{see for example~\cite[Lemma 2.8]{Pop}}]\label{lemma:J_1 x J_2}
Let $G$ and $H$ be Jordan groups. Then group $G \times H$ is Jordan, and $J(G \times H) = J(G)\cdot J(H)$.
\end{lemma}

Recall the standard definition.

\begin{definition}\label{def: char sub}
A subgroup $H$ of a group $G$ is called a \textit{characteristic} subgroup if for every automorphism $\varphi$ of $G$, one has $\varphi(H) = H$.
\end{definition}

The following theorem is useful for estimating Jordan constants of finite groups.

\begin{theorem}[{see for example~\cite[Theorem 1.41]{Isaacs}}]\label{theo: weak jordan constant} 
Let $G$ be a finite group, and $A$ be its abelian subgroup. Then there exists a characteristic abelian subgroup $N$ in $G$ such that $$[G:N]\leqslant[G:A]^2.$$
\end{theorem}

Let us prove an auxiliary proposition from group theory.

\begin{propos}\label{prop: section and direct product}
    Let $H$ be a group with a trivial center. Suppose we have a short exact sequence of groups $$1 \xrightarrow{} H \xrightarrow{} G \xrightarrow{} \mathbb{Z}/m\mathbb{Z} \xrightarrow{} 0,$$ and there exists an element $g\in G$ such that $g$ maps to $1$, and conjugation by $g$ induces an inner automorphism of $H$. Then $G \simeq H \times \mathbb{Z}/m\mathbb{Z}$.
\end{propos}

\begin{proof}
    Denote the homomorphism from $G$ to $\mathbb{Z}/m\mathbb{Z}$ by $p$. Then we have~$p(g) = 1$. Let us denote by $\alpha$ the automorphism of the group $H$ induced by conjugation by $g$. By the condition, $\alpha$ is an inner automorphism, so there exists an element $h\in H$ such that $\alpha$ is a conjugation by~$h$.
    
    Denote $g' = gh^{-1}$. Firstly, note that $p(g') = 1$, hence~$p((g')^m) =0$, that is, $(g')^m \in H$. Secondly, note that conjugation by~$g'$ induces a trivial automorphism of $H$, so conjugation by $(g')^m$ induces a trivial automorphism of $H$. Therefore the element $(g')^m$ lies in the center of $H$, which is trivial. It follows that $(g')^m = e$, hence the homomorphism $p$ has a section~$s:1\mapsto g'$. Thus, $G\simeq H \rtimes\mathbb{Z}/m\mathbb{Z}$ with trivial action, that is~\mbox{$G\simeq H \times\mathbb{Z}/m\mathbb{Z}$}. \end{proof}

The following proposition is a direct corollary of Proposition~\ref{prop: section and direct product}.

\begin{propos}\label{prop: C=0, Inn = Aut}
Let $H$ be a group with a trivial center. Suppose that all automorphisms of $H$ are inner. Let $A$ be a finite abelian group. Then any group $G$ which includes in the exact sequence $$1\xrightarrow{} H \xrightarrow{} G \xrightarrow{} A \xrightarrow{} 0,$$ is isomorphic to the direct product $H\times A$.
\end{propos}
\begin{proof}
Since $A$ is a finite abelian group, there is an isomorphism $$A \simeq \mathbb{Z}/n_{1}\mathbb{Z} \times \mathbb{Z}/n_2\mathbb{Z} \times \ldots \times \mathbb{Z}/n_r\mathbb{Z}.$$ Then we will denote elements of $A$  by $(m_1,m_2,\ldots, m_r)$, where $m_i \in \mathbb{Z}/n_i\mathbb{Z}$.

Denote the homomorphism from $G$ to $A$ by $p$. Let us choose elements~$g_1, g_2, \dots, g_r\in G$ such that $$p(g_i) = (0,\dots,0,\underset{i}{1},0,\dots,0).$$ Let us act on $H$ by conjugation by element~$g_i$. This action induces an automorphism 
~$\alpha$ of~$H$, and $\alpha$ is inner, since all automorphisms of the group $H$ are inner. So we are in the case of Proposition~\ref{prop: section and direct product}. Therefore, over each of the specified cyclic subgroups there is a section $$s_i: (0,\dots,0,\underset{i}{1},0,\dots,0) \mapsto g'_i,$$ and conjugation by element $g'_i$ induces a trivial automorphism of $H$.

Let us show that obtained sections are glued into a section over the entire group~$A$. To do this, it is enough to show that the elements $g'_i$ and $g'_j$ commute for all~$i, j\in\{1,\dots,n\}$. Consider the commutator $$c_{ij} = g'_ig' _j(g'_i)^{-1}(g' _j)^{-1}$$ of elements $g'_i$ and ~$g'_j$. Firstly, conjugation by this element induces a trivial automorphism of $H$. Secondly, this element lies in $H$, since $p(c_{ij}) = (0,\dots,0).$ But the center of the group $H$ is trivial, hence $c_{ij} = e$, that is, the elements~$g'_i$ and $g' _j$ commute. Therefore, we get a section $$s: A \xrightarrow{} G, \quad s: (m_1,m_2,\dots, m_r) \mapsto (g'_1)^{m_1}(g'_2)^{m_2}\dots(g'_r)^{m_r},$$ and $G \simeq H \rtimes A$. But, as we have already mentioned, conjugation by the element $g'_i$ induces a trivial automorphism of $H$ for any $i\in\{1,\dots,r\}$, which means that $G\simeq H\times A$.\end{proof}

We will need standard facts about automorphisms of groups $\mathfrak{S}_n$ and $\mathfrak{A}_n$.

\begin{theorem}[{see for example~\cite[\S 4.4, Exercise 18]{Dum-Foo}}]\label{theo:auto S_n}
Let $n$ be a positive integer,~$n\geqslant 3$, $n\neq 6$. Then $$\Aut{\mathfrak{S}_n} \simeq \mathfrak{S}_n.$$
\end{theorem}

To prove a similar result for the group $\mathfrak{A}_n$, we need the following simple lemma.

\begin{lemma}\label{lemma: conjugacy classes in A_n} 
Let $n$ be a positive integer. Let $C_g$ be the conjugacy class of an even permutation $g\in\mathfrak{S}_n$. Then
\begin{itemize}
    \item class $C_g$ splits into two conjugacy classes in $\mathfrak{A}_n$ if and only if the permutation $g$ decomposes into independent cycles of odd lengths, and all lengths are different (here a fixed point is considered as a cycle of length $1$);
    \item the class $C_g$ is a conjugacy class in $\mathfrak{A}_n$ if and only if the decomposition of $g$ into independent cycles contains a cycle of even length or two cycles of the same odd length.
\end{itemize}
\end{lemma}

\begin{proof}
    A simple exercise.
\end{proof}

\begin{theorem}\label{theo:auto A_n}
Let $n$ be a positive integer, $n\geqslant 4$, $n\neq 6$. Then $$\Aut{\mathfrak{A}_n} \simeq \mathfrak{S}_n.$$
\end{theorem}

\begin{proof}
Immediately note that $\mathfrak{S}_n$ is embedded in $\Aut{\mathfrak{A}_n}$ for $n\geqslant 4$. Indeed, consider the homomorphism $$\rho :\mathfrak{S}_n\xrightarrow{}\Aut{\mathfrak{A}_n}, \quad\tau \mapsto\rho_\tau,$$ where $\rho_\tau$ is a conjugation by permutation $\tau$. This homomorphism is injective because the centralizer of $\mathfrak{A}_n$ in $\mathfrak{S}_n$ is trivial for $n\geqslant 4$. Let us show that for $n\neq 6$ the homomorphism~$\rho$ is also surjective.

Let $\varphi$ be an arbitrary automorphism of $\mathfrak{A}_n$. Let us show that $\varphi$ maps cycles of length~$3$ into cycles of length $3$. Since $\varphi$ preserves the orders of elements, then triple cycle must maps into an element of order $3$, that is, into the product of $k$ pairwise disjoint triple cycles, for some $k\in\mathbb{Z}_{> 0}$. Note that for $n\leqslant 5$ we automatically have $k=1$, so it remains to deal with the case when~$n\geqslant 7$.

Suppose $n\geqslant 7$. Since $\varphi$ is an automorphism, the conjugacy classes maps into conjugacy classes. By Lemma~\ref{lemma: conjugacy classes in A_n}, all triple cycles form one conjugacy class in $\mathfrak{A}_n$. The products of $k$ pairwise disjoint triple cycles form one conjugacy class in $\mathfrak{A}_n$ by the same lemma. Equate the number of elements in these classes: $$2 {{n}\choose{3}} = \frac{n!}{k!3^k(n-3k)!}.$$ Taking into account the restriction of $n\geqslant 7$, the obtained equality is true only for $k=1$. Thus, we proved that for $n\neq 6$, the automorphism $\varphi$ maps cycles of length $3$ into cycles of length $3$.

Consider the following set of generators of $\mathfrak{A}_n$: $$A = \{ (123), (124), (125), \dots , (12n)\}.$$ Note that the product of any two considered permutations has the order $2$, which means that the same is true for the set of permutations: $$B = \{\varphi((123)),\varphi((124)), \varphi((125)), \dots, \varphi((12n))\}.$$ Let $\tau_1$ and $\tau_2$ be cycles of length $3$. It is easy to see that the order of permutation $\tau_1\circ\tau_2$ is equal to~$2$ if and only if these permutations have the form: $$\tau_1 = (ijk),\; \tau_2 = (ijl),\; k\neq l.$$ It follows that any pair of permutations from the set $B$ is represented in this form. Therefore, the entire set is represented as: $$B = \{(i_1i_2i_3), (i_1i_2i_4), (i_1i_2i_5), \dots, (i_1i_2i_n)\},$$ where $i_r \neq i_s$ for $r\neq s$. Consider a permutation $\mu\in\mathfrak{S}_n$ such that $$\mu(j)= i_j,\; j\in\{1,\dots,n\}.$$ Then for any permutation $\sigma\in A$ the equality $\varphi(\sigma) = \mu \sigma \mu^{-1}$ holds. Since the set $A$ generates $\mathfrak{A}_n$, the automorphism $\varphi$ coincides with the automorphism $\rho_\mu$. So the homomorphism~$\rho$ is surjective, and therefore is an isomorphism. \end{proof}

From Proposition~\ref{prop: C=0, Inn = Aut} we obtain a corollary.

\begin{sled}\label{sled: ext of S_n}
Let $n$ be a positive integer, $n\geqslant 3$, $n\neq 6$. Let $A$ be a finite abelian group. Then any group $G$ which includes in the exact sequence $$1\xrightarrow{}\mathfrak{S}_n\rightarrow{} G \xrightarrow{} A \xrightarrow{} 0,$$ is isomorphic to the direct product $\mathfrak{S}_n\times A$.
\end{sled}

\begin{proof}
For the specified $n$, the group $\mathfrak{S}_n$ has a trivial center and all its automorphisms are inner according to Theorem~\ref{theo:auto S_n}. So we can apply Proposition~\ref{prop: C=0, Inn = Aut}.
\end{proof}

\begin{propos}\label{prop: A_n x Z2 or S_n}
Let $n$ be a positive integer, $n\geqslant 4$, $n\neq 6$. Then the group $G$, which includes in the exact sequence  $$1 \xrightarrow{} \mathfrak{A}_n \xrightarrow{} G \xrightarrow{} \mathbb{Z}/2\mathbb{Z} \xrightarrow{} 0,$$ is isomorphic either to $\mathfrak{A}_n \times \mathbb{Z}/2\mathbb{Z}$, or to $\mathfrak{S}_n$. \end{propos}

\begin{proof}
Denote the homomorphism from $G$ to $\mathbb{Z}/2\mathbb{Z}$ by $p$. Choose an element~$g\in G$ such that $p(g) = 1$. Let us act on $\mathfrak{A}_n$ by conjugation by~$g$. This action induces the automorphism~$\alpha$ of the group~$\mathfrak{A}_n$. Now, if $\alpha$ is an inner automorphism, then by the Proposition~\ref{prop: section and direct product} we have the isomorphism $G\simeq\mathfrak{A}_n\times\mathbb{Z}/2\mathbb{Z}$.

Assume that $\alpha$ is not an inner automorphism. By Theorem~\ref{theo:auto A_n} we have an isomorphism~$\Aut(\mathfrak{A}_n)\simeq\mathfrak{S}_n$, so we can choose an element $h_0\in\mathfrak{A}_n\subset G$ such that the automorphism $\beta \in\Aut(\mathfrak{A}_n)$, induced by conjugation by element~$g_0 = gh_0$, is a conjugation by transposition. Firstly, we have $p(g_0) = 1$. Hence~$p(g^2_0) = 0$, that is, $g^2_0 \in\mathfrak{A}_n$. Secondly, conjugation by the element $g^2_0$ induces a trivial automorphism of $\mathfrak{A}_n$, which means that~$g^2_0 =e$. Therefore, the homomorphism $p$ has a section $s: 1 \mapsto g_0$, and $G\simeq\mathfrak{A}_n\rtimes\mathbb{Z}/2\mathbb{Z}$, where nontrivial element of the group $\mathbb{Z}/2\mathbb{Z}$ acts by conjugation by transposition. It follows that $G\simeq\mathfrak{S}_n$.\end{proof}

\begin{propos}\label{prop: A_4 x Zm}
Let $n$ and $m$ be positive integers, $n\geqslant 4$, $n\neq 6$. Let the group $G$ be included in the exact sequence $$1 \xrightarrow {} \mathfrak {A} _n \xrightarrow {} G\xrightarrow {} \mathbb {Z} / m \mathbb{Z} \xrightarrow{} 0.$$ If $m = 2k + 1$, then $G$ is isomorphic to $\mathfrak{A}_n\times\mathbb{Z}/m\mathbb{Z}$. If $m = 2k$, then either $G$ is isomorphic to~$\mathfrak{A}_n\times\mathbb{Z}/m\mathbb{Z}$, or at least contains a normal subgroup isomorphic to $\mathfrak{A}_n\times\mathbb{Z}/k\mathbb{Z}$.
\end{propos}

\begin{proof}
Denote the homomorphism from $G$ to $\mathbb{Z}/m\mathbb{Z}$ by $p$. Consider an element $g\in G$ such that~$p(g) = 1$. Let us act on $\mathfrak{A}_n$ by conjugating by~$g$. This action induces the automorphism $\alpha$ of~$\mathfrak{A}_n$. By Theorem~\ref{theo:auto A_n}, the automorphism~$\alpha$ is a conjugation by some permutation $\sigma\in\mathfrak{S}_n$.

Let $m$ be odd at first. Since $p(g^m) = 0$, then $g^m$ lies in $\mathfrak{A}_n$, hence~$\alpha^m$~is an inner automorphism, that is, conjugation by an even permutation. Therefore, $\sigma$ is also an even permutation and $\alpha$ is also an inner automorphism. According to Proposition~\ref{prop: section and direct product}, we have an isomorphism $G\simeq\mathfrak{A}_n\times\mathbb{Z}/m\mathbb{Z}$.

Now let $m$ be even, that is, $m = 2k$. If $\sigma$ is an even permutation, then $\alpha$ is an inner automorphism, and according to Proposition~\ref{prop: section and direct product} we have an isomorphism $G \simeq \mathfrak{A}_n \times \mathbb{Z}/m\mathbb{Z}$. If $\sigma$ is an odd permutation, then conjugation by the element $g^2$ already induces an inner automorphism of $\mathfrak{A}_n$. Denote $G' = \langle \mathfrak{A}_n, g^2\rangle$, then we get into conditions of Proposition~\ref{prop: section and direct product} for the exact sequence $$1 \xrightarrow{} \mathfrak{A}_n \xrightarrow{} G' \xrightarrow{} \mathbb{Z}/k\mathbb{Z} \xrightarrow{} 0.$$ Therefore, $G'$ is isomorphic to $\mathfrak{A}_n\times\mathbb{Z}/k\mathbb{Z}$. Also, $G'$ is normal in $G$, since it has index~$2$.\end{proof}

We will also need the following presentation of the group $\mathfrak{A}_5$.

\begin{lemma}\label{lemma: A_5 presentation} 
Consider a group given by generators and relations: $$G=\langle x, y\mid x^5= y^2=(xy)^3 = e\rangle.$$ Then $G$ is isomorphic to the group $\mathfrak{A}_5$, and there is an isomorphism which maps $x$ to the permutation $(12345)$, and $y$ to the permutation $(12)(34)$.
\end{lemma}

\begin{proof}
In the example~\cite[Kapitel I, Beispiel 19.9]{Huppert}, it is proved that $G$ is isomorphic to~$\mathfrak{A}_5$. Now note that permutations $(12345)$ and $(12)(34)$ generate the group~$\mathfrak{A}_5$ and satisfy the conditions $$(12345)^5=e,\; ((12)(34))^2 = e,\;((12345)(12)(34))^3 = e.$$ Therefore, the specified isomorphism exists.\end{proof}

\section{Semidirect products}\label{sect: semidirect 1}

In this section, we study the Jordan constant of groups of the form $(\Gamma\times\Gamma)\rtimes\mathbb{Z}/2\mathbb{Z}$, where the nontrivial element of the group $\mathbb{Z}/2\mathbb{Z}$ acts by permutation of factors. The element of the group $(\Gamma \times \Gamma) \rtimes \mathbb{Z}/2\mathbb{Z}$ we write as $g = (g_1, g_2, i)$, where $g_1, g_2 \in \Gamma$, and~$i = 0$ (a trivial element of $\mathbb{Z}/2\mathbb{Z}$), or $i = 1$ (a nontrivial element of $\mathbb{Z}/2\mathbb{Z}$).

Let $g = (g_1, g_2, i), h = (h_1, h_2, j) \in (\Gamma \times \Gamma) \rtimes \mathbb{Z}/2\mathbb{Z}$, then the group operation looks as follows: $$gh = \begin{cases}
   (g_1h_1, g_2h_2, i+j),\; \text{if}\; i = 0;\\
   (g_1h_2, g_2h_1, i+j),\; \text{if}\; i = 1.
 \end{cases} $$

The following three technical lemmas will be needed to prove Theorem~\ref{theo:theo P^1 x P^1}.

\begin{lemma}\label{lemma: JxJsemiZ2}
Let $\Gamma$ be a Jordan group. Then $(\Gamma\times\Gamma) \rtimes\mathbb{Z}/2\mathbb{Z}$ is Jordan (a nontrivial element of the group $\mathbb{Z}/2\mathbb{Z}$ acts by permutation of factors) and the Jordan constant is reached on the group $H$, which is included in the exact sequence $$1 \xrightarrow{} G_1 \times G_2 \xrightarrow{} H \xrightarrow{} \mathbb{Z}/2\mathbb{Z} \xrightarrow{} 0,$$ where $G_1$ and $G_2$ are isomorphic subgroups of $\Gamma$. \end{lemma}

\begin{proof}
Group $(\Gamma \times \Gamma) \rtimes \mathbb{Z}/2\mathbb{Z}$ is obviously Jordan. Let $G \subset (\Gamma \times \Gamma) \rtimes \mathbb{Z}/2\mathbb{Z}$ be a finite subgroup on which the Jordan constant is reached, that is $$J(G) = J((\Gamma \times \Gamma) \rtimes \mathbb{Z}/2\mathbb{Z}).$$ Our goal is to find a finite subgroup $H$ of the required form, with $J(H) \geqslant J(G)$ (note that this condition immediately implies the equality $J(H) = J(G)$, since $J(G)$ equals to the Jordan constant of the entire group). Let $p_1$ and $p_2$ be projections of $\Gamma\times\Gamma$ onto the first and the second factors, respectively.

Let us assume that $G\subset\Gamma\times\Gamma$. Denote $G' = p_1(G)$ and $G'' = p_1(G)$. Then we have an inclusion $G \subset G' \times G''$ which implies an inequality $$J(G')J(G'') = J(G' \times G'') \geqslant J(G).$$ It follows that either $J(G') \geqslant \sqrt{J(G)}$, or $J(G'') \geqslant \sqrt{J(G)}$. Without loss of generality, we can assume that the first case holds, then we can take $H$ equal to $(G'\times G')\rtimes\mathbb{Z}/2\mathbb{Z}$. Indeed, for this group we have $$J(H) \geqslant J(G')^2 \geqslant J(G).$$

Now assume that $G \not\subset \Gamma \times \Gamma$. Denote by  $G^0$ its intersection with $\Gamma \times \Gamma$ and denote projections~\hbox{$G_1 = p_1(G^0)$} and $G_2 = p_2(G^0)$. 
Note that $G_1$ and $G_2$ are conjugate in $\Gamma$. Indeed, $G^0\subset G$ is a normal subgroup. Conjugating $G^0$ by the element $$\gamma = (\gamma_1,\gamma_2,1) \in G\backslash G^0,$$ we obtain: $$G_1 = p_1(G^0) = p_ 1(\gamma G^0\gamma^{-1}) = \gamma_1p_2(G^0)\gamma_1^{-1} = \gamma_1G_2\gamma_1^{-1}.$$ Similarly we have $G_2 = \gamma_2G_1\gamma_2^{-1}$. Then we can take $H$ equals to $\langle G_1 \times G_2, \gamma \rangle$. Indeed, $H$ is included in the exact sequence $$1 \xrightarrow{} G_1 \times G_2 \xrightarrow{} H \xrightarrow{} \mathbb{Z}/2\mathbb{Z} \xrightarrow{} 0,$$ where $G_1$ and $G_2$ are isomorphic, since they are conjugate in $\Gamma$. Also $G$ is a subgroup of~$H$, therefore $J(H) \geqslant J(G)$. \end{proof}

\begin{lemma}\label{lemma:GxGsemiZ2} Let $G$ be a nontrivial finite group. Then $$ J((G\times G) \rtimes \mathbb{Z}/2\mathbb{Z}) = 2J(G)^2,$$ where the nontrivial element of the group $\mathbb{Z}/2\mathbb{Z}$ acts by permutation of factors.
\end{lemma}

\begin{proof}
Firstly, assume that $J(G) = |G|$. In this case, there are no nontrivial normal abelian subgroups in $G$. Since $G$ is a non-trivial group, the specified semidirect product is not a direct one, and it is easy to see that in this case the group $(G \times G) \times \mathbb{Z}/2\mathbb{Z}$ also does not contain non-trivial normal abelian subgroups, that is, $$J((G \times G) \times \mathbb{Z}/2\mathbb{Z}) = 2J(G)^2.$$

Now assume that $J(G) \neq |G|$. Let $A\subset G$ be a normal abelian subgroup such that~\mbox{$[G:A]= J(G)$}. Then $$A\times A\subset(G\times G)\rtimes\mathbb{Z}/2\mathbb{Z}$$ is a normal abelian subgroup of index $2J(G)^2$. It remains to show that there are no normal abelian subgroups of smaller index.

Let $H \subset (G \times G) \rtimes \mathbb{Z}/2\mathbb{Z}$ be a normal abelian subgroup. If $H$ is contained in $G\times G$, then we have $[G\times G :H] \geqslant J(G)^2$ by Lemma~\ref{lemma:J_1 x J_2}. Therefore, $$[(G \times G) \rtimes \mathbb{Z}/2\mathbb{Z} : H] \geqslant 2J(G)^2.$$

If $H$ is not contained in $G\times G$, then we denote by $H^0$ the intersection of $H$ and~$G\times G$. Then $H^0$ is a normal abelian subgroup of $G\times G$ and $H=\langle H^0, g\rangle$, where~\mbox{$g=(g_1,g_2,1)\in H\backslash H^0.$} Let $p_1$ and $p_2$ be projections of $G\times G$ on the first and the second factors, respectively. Denote $H_1 = p_1(H^0)$ and $H_2=p_2(H^0)$. Then~$H_1$ and~$H_2$ are normal abelian subgroups in $G$, thus $[G:H_1] \geqslant J(G)$ and $[G:H_2] \geqslant J(G)$.

Since $H$ is abelian, then any element $$h = (h_1,h_2,0) \in H,$$ where $h_1 \in H_1, h_2 \in H_2$, commutes with $g$: $$(h_1,h_2,0)(g_1,g_2,1) = (g_1,g_2,1)(h_1,h_2,0).$$ After multiplication we get $$(h_1g_1,h_2g_2,1)=(g_1h_2,g_2h_1,1).$$ Therefore
$h_1 = g_1h_2g_1^{-1}$ and $h_2 = g_2h_1g_2^{-1}$ and the map $h_1 \mapsto (h_1,g_2h_1g_2^{-1},0)$ defines an isomorphism between $H_1$ and~$H^0$. Then we have $$[(G \times G) \rtimes \mathbb{Z}/2\mathbb{Z}:H] = [G \times G: H^0] = [G:H_1]\cdot|G| \geqslant J(G)\cdot|G|.$$ Since $J(G)\neq|G|$, then $|G|\geqslant 2J(G)$, and we get an estimate for the index $$[(G \times G) \rtimes \mathbb{Z}/2\mathbb{Z}:H] \geqslant 2J(G)^2.$$ 

As a result, we presented a normal abelian subgroup of index $2J(G)^2$ and showed that there are no normal abelian subgroups of smaller index, thereby the lemma is proved. \end{proof}

Now we prove the main lemma of this section, which we will apply in the proof of Theorem~\ref{theo:theo P^1 x P^1}.

\begin{lemma}\label{lemma:Gamma X Gamma X Z/2Z}
    Let $\Gamma$ be a Jordan group, containing a nontrivial finite subgroup. Then $$J((\Gamma \times \Gamma) \rtimes \mathbb{Z}/2\mathbb{Z}) = 2J(\Gamma)^2.$$
\end{lemma}

\begin{proof}
For finite groups, this assertion is proved in Lemma~\ref{lemma:GxGsemiZ2}. Let~$\Gamma$ be an infinite group.
First, let us show that $$J((\Gamma \times \Gamma) \rtimes \mathbb{Z}/2\mathbb{Z}) \geqslant 2J(\Gamma)^2.$$ To do this, we need to find a finite subgroup of $(\Gamma \times \Gamma) \rtimes \mathbb{Z}/2\mathbb{Z}$ with Jordan constant equal to~$2J(\Gamma)^2$.

Consider finite subgroup $G \subset \Gamma$ on which the Jordan constant is reached, that is~\mbox{$J(G) = J(\Gamma)$} (we can assume, that $G$ is nontrivial, since there exist nontrivial finite subgroups in $\Gamma$). Denote $$\tilde G = (G \times G) \rtimes \mathbb{Z}/2\mathbb{Z} \subset (\Gamma \times \Gamma) \rtimes \mathbb{Z}/2\mathbb{Z}.$$ We have $J(\tilde G) = 2J(G)^2 = 2J(\Gamma)^2$ by Lemma~\ref{lemma:GxGsemiZ2}.

Now let us show, that there are no finite subgroups with larger Jordan constant. By Lemma~\ref{lemma: JxJsemiZ2}, it is enough to prove this for finite subgroups $H$ included in the exact sequence $$1 \xrightarrow{} G_1 \times G_2 \xrightarrow{} H \xrightarrow{} \mathbb{Z}/2\mathbb{Z} \xrightarrow{} 0.$$

Denote the homomorphism from $H$ to $\mathbb{Z}/2\mathbb{Z}$ by $p$. Choose an element~$\gamma\in H$ such that~$p(\gamma) = 1$. Then $\gamma$ can be written as~$$\gamma = (\gamma_1, \gamma_2, 1).$$ Let us conjugate the normal subgroup $G_1\times G_2$ by~$\gamma$, then we obtain the equalities: $$G_ 1 = \gamma_1G_2\gamma_1^{-1}, \quad G_2 = \gamma_2G_1\gamma_2^{-1}.$$ Let $N_1$ be a normal abelian subgroup of $G_1$ on which the Jordan constant is reached. Then the Jordan constant of the group $G_2$ is reached on the subgroup~$N_2 = \gamma_2 N_1 \gamma_2^{-1}$. In particular, the following inequalities hold: $$[G_1:N_1] = J(G_1) \leqslant J(\Gamma), \quad [G_2:N_2] = J(G_2) \leqslant J(\Gamma).$$

The subgroup $N_1\times N_2$ is obviously normal in $G_1\times G_2$. Let us show that it is normal in~$H$. It is enough to check that $N_1\times N_2$ is normalized by the element $\gamma$. Take an element~\mbox{$n = (n_1, n_2, 0) \in N_1 \times N_2$} and conjugate by $\gamma$: 
$$\gamma n \gamma^{-1} =(\gamma_1, \gamma_2, 1)(n_1, n_2, 0)(\gamma^{-1}_2, \gamma^{-1}_1, 1) = (\gamma_1 n_2 \gamma_1^{-1}, \gamma_2 n_1 \gamma_2^{-1}, 0).$$ Note that $\gamma_2 n_1 \gamma_2^{-1} \in N_2$ by definition of a subgroup $N_2$. Also, by definition of a subgroup~$N_2$, there exist an element  $n'_1 \in N_1$ such that $n_2 = \gamma_2 n'_1 \gamma_2^{-1}$. Therefore $$\gamma_1 n_2 \gamma_1^{-1} = \gamma_1 \gamma_2 n'_1 \gamma_2^{-1} \gamma_1^{-1} \in N_1,$$ since $\gamma_1\gamma_2$ lies in $G_1$, since $\gamma^2 \in G_1 \times G_2$. Hence $$\gamma n \gamma^{-1} = (\gamma_1 n_2 \gamma_1^{-1}, \gamma_2 n_1 \gamma_2^{-1}, 0) \in N_1 \times N_2,$$ and $N_1 \times N_2$ is a normal abelian subgroup of $H$ of index $$[H:N_1\times N_2] = 2\cdot[G_1:N_1]\cdot[G_2:N_2] \leqslant 2J(\Gamma)^2.$$ Therefore $J(H)\leqslant 2J(\Gamma)^2$, and the lemma is proved. \end{proof}


\section{\boldmath{Jordan constants of a group $\text{Aut}(\mathbb{P}^1_K \times \mathbb{P}^1_K)$}}\label{sect:P1xP1}

The automorphism group of the surface $\mathbb{P}^1_K\times\mathbb{P}^1_K$ has an explicit description:$$\Aut(\mathbb{P}^1_K \times \mathbb{P}^1_K) \simeq (\PGL_2(K) \times \PGL_2(K)) \rtimes \mathbb{Z}/2\mathbb{Z},$$ where the nontrivial element of the group $\mathbb{Z}/2\mathbb{Z}$ acts by permutation of the factors. Therefore, to study the Jordan constants of the group $\text{Aut}(\mathbb{P}^1_K\times\mathbb{P}^1_K)$, it is necessary to understand which finite subgroups does group $\PGL_2(K)$ contain.

If the field $K$ is algebraically closed, it is well known that the finite subgroups of the group $\PGL_2(K)$ are $\mathbb{Z}/n\mathbb{Z}$, $D_{2n}\;($\text{for}\ $n\geqslant 2)$, $\mathfrak{A}_4$, $\mathfrak{S}_4$ and $\mathfrak{A}_5$. If $K$ is an arbitrary, all finite subgroups of the group $\PGL_2(K)$ occur in the above list because $\PGL_2(K)$ is a subgroup of $\PGL_2(\overline{K})$, however, no one guarantees that all groups in the list are realized as finite subgroups in $\PGL_2(K)$.

We are working with an arbitrary field $K$ of characteristic zero.
In this case, it is also well known which finite groups are realized as subgroups of~$\PGL_2(K)$ depending on the arithmetic properties of the field.

\begin{propos}[{\cite[Proposition 1.1]{Bea}}]\label{prop:Beauville}
Let $K$ be a field of characteristic $0$ and $\xi_m$ be a primitive $m$-th root of unity.
\begin{enumerate}
    \item $\emph{PGL}_2(K)$ contains {$\mathbb{Z}/m\mathbb{Z}$}, $D_{2m}$ if and only if $K$ contains~$\xi_m + \xi_m^{-1}$ $($in particular,  $\emph{PGL}_2(K)$ always contains $D_6$, $D_8$ и $D_{12}$$)$;
    \item $\emph{PGL}_2(K)$ contains $\mathfrak{A}_4$, $\mathfrak{S}_4$ if and only if $-1$ is a sum of two squares in $K$;
    \item $\emph{PGL}_2(K)$ contains $\mathfrak{A}_5$ if and only if $-1$ is a sum of two squares in $K$, and $K$ contains $\sqrt{5}$.
\end{enumerate}
\end{propos}

\begin{rem}\label{rem:D2n,A4,S4,A5}
We have $$J(\mathbb{Z}/n\mathbb{Z}) = 1,\; J(D_{4}) = 1,\; J(D_{2n}) = 2\; 
\text{for}\; n\geqslant 3,$$ $$ J(\mathfrak{A}_4) = 3,\; J(\mathfrak{S}_4) = 6, J(\mathfrak{A}_5) = 60.$$
In particular, it can be seen from here that if the group $\PGL_2(K)$ does not contain $\mathfrak{A}_5$, then the Jordan constant $J(\PGL_2(K))$ does not exceed $6$.

Also note that if the group $G$ is isomorphic to one of these groups, then there is a characteristic abelian subgroup $A\subset G$ such that $[G:A] = J(G)$. That is, for any finite subgroup $H\subset\PGL_2(K)$, the Jordan constant of the group $H$ is reached on a characteristic subgroup. \end{rem}

From Proposition~\ref{prop:Beauville} we obtain an obvious corollary about the Jordan constants of the group~$\PGL_2(K)$.

\begin{sled}\label{sled:J(PGL_2)}
Let $K$ be a field of characteristic $0$.
\begin{enumerate}
        \item $J(\PGL_2(K)) = 60$ if and only if $\sqrt{5} \in K$ and $-1$ is a sum of two squares in $K$; 
        \item $J(\PGL_2(K)) = 6$ if and only if $\sqrt{5} \not\in K$ and $-1$ is a sum of two squares in $K$;
        \item $J(\PGL_2(K)) = 2$ if and only if $-1$ is not a sum of two squares in $K$.
\end{enumerate}
\end{sled}

\begin{sled}\label{sled:J(PGL_2xPGL_2xZ2)}
Let $K$ be a field of characteristic $0$. Consider the semidirect product of groups $(\PGL_2(K)\times\PGL_2(K))\rtimes\mathbb{Z}/2\mathbb{Z}$, where the nontrivial element of the group $\mathbb{Z}/2\mathbb{Z}$ acts by permutation of factors.
\begin{enumerate}
        \item $J((\PGL_2(K) \times \PGL_2(K)) \rtimes \mathbb{Z}/2\mathbb{Z}) = 7200$ if and only if $\sqrt{5} \in K$ and $-1$ is a sum of two squares in $K$; 
        \item $J((\PGL_2(K) \times \PGL_2(K)) \rtimes \mathbb{Z}/2\mathbb{Z}) = 72$ if and only if $\sqrt{5} \not\in K$ and $-1$ is a sum of two squares in $K$;
        \item $J((\PGL_2(K) \times \PGL_2(K)) \rtimes \mathbb{Z}/2\mathbb{Z}) = 8$ if and only if $-1$ is not a sum of two squares in $K$.
\end{enumerate}
\end{sled} 
\begin{proof}
Apply Lemma~\ref{lemma:Gamma X Gamma X Z/2Z} to Corollary~\ref{sled:J(PGL_2)}.
\end{proof}

Now we can prove Theorem~\ref{theo:theo P^1 x P^1}.

\begin{proof}[proof of Theorem~\ref{theo:theo P^1 x P^1}.] 
The automorphism group of the surface $\mathbb{P}^1_K\times\mathbb{P}^1_K$ is isomorphic to the group $(\PGL_2(K)\times\PGL_2(K)) \rtimes\mathbb{Z}/2\mathbb{Z}$, where the nontrivial element of the group~$\mathbb{Z}/2\mathbb{Z}$ acts by permutation of factors. Therefore, applying  Corollary~\ref{sled:J(PGL_2xPGL_2xZ2)}, we obtain an assertion of the theorem. \end{proof}




\section{\boldmath{Jordan constants of groups of type $\PGL_2(L) \rtimes \mathbb{Z}/2\mathbb{Z}$}}\label{sect: semidirect 2}

In this section, we estimate and compute the Jordan constants of groups of the form $$\PGL_2(L)\rtimes\mathbb{Z}/2\mathbb{Z},$$ where $L$ is a field of characteristic zero, as usual.

An element of the group $\PGL_2(L) \rtimes\mathbb{Z}/2\mathbb{Z}$ we will write as $g = (\gamma, i)$, where~\mbox{$\gamma \in \PGL_2(L)$}, and $i = 0$ (the trivial element of the group $\mathbb{Z}/2\mathbb{Z}$), or $i = 1$ (the nontrivial element of the group $\mathbb{Z}/2\mathbb{Z}$).

Let $g_1 = (\gamma_1, i)$, $g_2 = (\gamma_2, j) \in\PGL_2(L) \rtimes\mathbb{Z}/2\mathbb{Z}$, then the group operation looks as following: $$g_1g_2 = (\gamma_1(\varphi_i(\gamma_2)), i +j)
, $$ where $\varphi_i$ is an automorphism of the group $\PGL_2(L)$ included in the definition of a semidirect product.

\begin{lemma}\label{lemma: |G| < 120}
Let $L$ be a field of characteristic $0$. Then $$J(\PGL_2(L) \rtimes \mathbb{Z}/2\mathbb{Z}) \leqslant120.$$
\end{lemma}

\begin{proof}
Let $G\subset\PGL_2(L)\rtimes\mathbb{Z}/2\mathbb{Z}$ be a finite subgroup. Denote by $G^0$ the intersection of $G$ with $\PGL_2(L)$. Then, according to Remark~\ref{rem:D2n,A4,S4,A5}, there is an abelian characteristic subgroup $A\subset G^0$, on which the Jordan constant of the group~$G^0$ is reached. Hence, $A$ is a normal abelian subgroup of $G$ of index $2J(G^0)$, and $$J(G)\leqslant 2J(G^0)\leqslant 2J(\PGL_2(L)) \leqslant 120.$$ Since $G$ is an arbitrary finite subgroup of $\PGL_2(L)\rtimes\mathbb{Z}/2\mathbb{Z}$, the lemma is proved.
\end{proof}

\begin{lemma}\label{lemma: sqrt(5) - a,b -} 
Let $L$ be such a field of characteristic $0$ that group $\PGL_2(L)$ does not contain a subgroup isomorphic to $\mathfrak{A}_5$. Then  $$J(\PGL_2(L) \rtimes \mathbb{Z}/2\mathbb{Z}) \leqslant 6.$$
\end{lemma}

\begin{proof}
Since $\PGL_2(L)$ does not contain a subgroup isomorphic to $\mathfrak{A}_5$, then we have the inequality $J(\PGL_2(L)\leqslant 6$ according to Remark~\ref{rem:D2n,A4,S4,A5}. It remains to show that the Jordan constants of the finite subgroups contained in $\PGL_2(L)\rtimes\mathbb{Z}/2\mathbb{Z}$, but not contained in~$\PGL_2(L)$, also do not exceed 6.

Let $G$ be a finite subgroup of $\PGL_2(L)\rtimes\mathbb{Z}/2\mathbb{Z}$ not contained in $\PGL_2(L)$. Then~$G$ is included in the exact sequence $$1\xrightarrow{} G^0 \xrightarrow{} G\xrightarrow{} \mathbb{Z}/2\mathbb{Z}\xrightarrow{} 0,$$ where $G^0$ is the intersection of the group $G$ with $\PGL_2(L)$. Since $\PGL_2(L)$ does not contain $\mathfrak{A}_5$, then, according to Proposition~\ref{prop:Beauville}, the group $G^0$ is isomorphic to either a cyclic group, or a dihedral group, or $\mathfrak{A}_4$, or $\mathfrak{S}_4$. 

Suppose $G^0$ is isomorphic to either $\mathbb{Z}/n\mathbb{Z}$ or $D_{2n}$ or $\mathfrak{A}_4$. By Remark~\ref{rem:D2n,A4,S4,A5} it contains characteristic abelian subgroup $A$ of index at most $3$. Then $A$ is normal abelian subgroup in $G$ of index at most $6$, and $J(G) \leqslant 6$.

Suppose $G^0 \simeq \mathfrak{S}_4$. Applying Corollary~\ref{sled: ext of S_n} and Remark~\ref{rem:D2n,A4,S4,A5}, we get $J(G)\leqslant 6$.

Thus, we have shown that the Jordan constant does not exceed 6 for all possible finite subgroups in $\PGL_2(L)\rtimes\mathbb{Z}/2\mathbb{Z}$, and the lemma is proved.\end{proof}

The following two propositions are useful when the base field does not contain~$\sqrt{5}$.

\begin{propos}\label{prop:A_5 action}
    Let $K$ be a field of characteristic $0$, and $r \in K$ be an element, which is not a square. Assume that $K(\sqrt{r})$ contains $\sqrt{5}$, and there exist $a, b, c, d \in K$, such that $$(a + b\sqrt{r})^2 + (c + d\sqrt{r})^2 = -1.$$ Consider two matrices in $\PGL_2(K(\sqrt{r}))$:
    $$A = \begin{pmatrix}
  0 & 1\\
  -1 & 0\\
  \end{pmatrix},\;
  C = \begin{pmatrix}
  2c + 2d\sqrt{r} + \sqrt{5} - 3& 2a + 2b\sqrt{r} - \sqrt{5} + 1\\
  2a + 2b\sqrt{r} + \sqrt{5} - 1& -2c - 2d\sqrt{r} + \sqrt{5} -3\\
  \end{pmatrix}. $$
  The group $G' = \langle A,C\rangle$ is isomorphic to $\mathfrak{A}_5$, and there exists an isomorphism that maps matrix $A$ to permutation $(12)(34)$, and matrix $C$ to permutation $(12345)$. \end{propos}

\begin{proof}
Taking in account Lemma~\ref{lemma: A_5 presentation}, it is enough to check that the matrices $A$ and $C$ satisfy the relations $A^2 = e$, $C^5 = e$ and $(CA)^3 = e$. Direct calculations show that this is the case:
  $$A^2 = \begin{pmatrix}
    -1 & 0 \\
    0 & -1 \\
\end{pmatrix} = e;$$ 
$$C^5 =\begin{pmatrix}
    -2560\sqrt{5} + 5632& 0 \\
    0 &  -2560\sqrt{5} + 5632\\
\end{pmatrix} =  e;$$ 
$$(CA)^3 =\begin{pmatrix}
    -64\sqrt{5} + 128 & 0 \\
    0 & -64\sqrt{5} + 128 \\
\end{pmatrix} = e.$$
It follows that $G'$ is isomorphic to the quotient of the group $\mathfrak{A}_5$. But $\mathfrak{A}_5$ is simple, and~$G'$ is nontrivial, thus $G'\simeq \mathfrak{A}_5$. Existence of a specified isomorphism is guaranteed by Lemma~\ref{lemma: A_5 presentation}. \end{proof}

\begin{lemma}\label{lemma: S_5 action}
Let $K$ be a field of characteristic $0$. Assume that $\sqrt{5}\notin K$, and $-1$ is the sum of two squares in $K(\sqrt{5})$. Then $$J(\PGL_2(K(\sqrt{5})) \rtimes\mathbb{Z}/2\mathbb{Z}) = 120,$$ where the nontrivial element of the group $\mathbb{Z}/2\mathbb{Z}$ acts by the Galois involution of the extension~$K\subset K(\sqrt{5})$.
\end{lemma}

\begin{proof}
From Lemma~\ref{lemma: |G| < 120} we have the inequality $$J(\PGL_2(K(\sqrt{5})) \rtimes\mathbb{Z}/2\mathbb{Z})\leqslant 120.$$ It remains to find a finite subgroup with Jordan constant equal to 120.

By the assumption of the lemma, there are $a, b, c, d \in K$, such that $$(a + b\sqrt{5})^2 + (c + d\sqrt{5})^2 = -1.$$ Consider the matrix $$R = \begin{pmatrix}
  a+c & a-c\\
  a-c & -a -c\\
  \end{pmatrix} \in \PGL_2(K(\sqrt{5})).$$ Let $G' \subset \PGL_2(K(\sqrt{5}))$ be a group from Proposition~\ref{prop:A_5 action} for $r = 5$. Let us show that $$\postdisplaypenalty 10000 \langle (G',0), (R,1) \rangle \simeq \mathfrak{S}_5,$$ and that's where the proof ends, because $J(\mathfrak{S}_5) = 120$.

Firstly, note that the element $(R,1)$ has order $2$. Secondly, the subgroup $(G',0)$ is invariant with respect to the conjugation by the element $(R,1)$. Indeed, it is enough to check that generators remain in the group after conjugation: $$(R,1)(A,0)(R,1) = (A,0);$$ $$(R,1)(C,0)(R,1) = ((C^2A)^3,0).$$ 
Therefore $$\langle (G',0), (R,1) \rangle\simeq\mathfrak{A}_5 \rtimes\mathbb{Z}/2\mathbb{Z},$$ moreover, the product is not direct, since the conjugation by element $(R,1)$ induces a noninner automorphism of the group $(G',0)$. Thus, according to Proposition~\ref{prop: A_n x Z2 or S_n}, we have the isomorphism $\langle (G',0), (R,1)\rangle\simeq\mathfrak{S}_5$. \end{proof}

Observe that for an element $T \in \PGL_2(K)$ the value $\frac{\tr^2(T)}{\det(T)}$ is well-defined and invariant under conjugation in $\PGL_2(K)$. Let us prove the following auxiliary lemma.

\begin{lemma}\label{lemma: tr/det}
Let $B \in \PGL_2(K)$ be an element of order $5$. Then $$\frac{\tr^2(B)}{\det(B)} = \frac{3}{2} \pm \frac{1}{2}\sqrt{5}, \; \frac{\tr^2(B^2)}{\det(B^2)} = \frac{3}{2} \mp \frac{1}{2}\sqrt{5}.$$
\end{lemma}

\begin{proof}
    Over the algebraic closure $\overline{K}$, matrix $B$ is diagonalizable, namely there exists an element $C \in \PGL_2(\overline{K})$ such that $$CBC^{-1} = \begin{pmatrix}
        \delta_1 & 0\\
        0 & \delta_2\\
    \end{pmatrix}.$$
Since $B$ has order $5$, then we have an equality $\delta_2 = \xi\delta_1$, where $\xi$ is a $5$-th root of unity. Therefore $$\frac{\tr^2(B)}{\det(B)} = \frac{\tr^2(CBC^{-1})}{\det(CBC^{-1})} = \frac{(\delta_1 + \xi\delta_1)^2}{\xi\delta_1^2} = \xi^{-1} + 2 + \xi = \frac{3}{2} \pm \frac{1}{2}\sqrt{5}$$ and $$\frac{\tr^2(B^2)}{\det(B^2)} = \frac{\tr^2(CB^2C^{-1})}{\det(CB^2C^{-1})} = \frac{(\delta^2_1 + \xi^2\delta^2_1)^2}{\xi^2\delta_1^4} = \xi^{-2} + 2 + \xi^2 = \frac{3}{2} \mp \frac{1}{2}\sqrt{5}.$$ 
\end{proof}

The last lemma of this section handles the case of such fields containing~$\sqrt{5}$, that $-1$ is not a sum of two squares in this fields.

\begin{lemma}\label{lemma: J = 60 for every quadratic extension}
Let $K$ be a field of characteristic $0$ such that $-1$ is not a sum of two squares in $K$, but $\sqrt{5}$ lies in $K$. Let $K\subset L$ be a quadratic extension of fields. Consider the group~$\PGL_2(L)\rtimes\mathbb{Z}/2\mathbb{Z}$, where the nontrivial element of the group $\mathbb{Z}/2\mathbb{Z}$ acts by Galois involution of the extension $K\subset L$. Then $$J(\PGL_2(L)\rtimes\mathbb{Z}/2\mathbb{Z}) \leqslant 60.$$
\end{lemma}

\begin{proof}
Suppose $-1$ is not a sum of two squares in $L$. Then, by Proposition~\ref{prop:Beauville}, the group $\PGL_2(L)$ does not contain a subgroup isomorphic to $\mathfrak{A}_5$. Therefore, by Lemma~\ref{lemma: sqrt(5) - a,b -} we have $$J(\PGL_2(L)\rtimes\mathbb{Z}/2\mathbb{Z})\leqslant 6.$$

Now suppose $-1$ is a sum of two squares in $L$. Note that in this case, according to Proposition~\ref{prop:Beauville}, the group $\PGL_2(L)$ contains a finite subgroup isomorphic to $\mathfrak{A}_5$.

Let $H$ be an arbitrary finite subgroup of $\PGL_2(L)\rtimes\mathbb{Z}/2\mathbb{Z}$. If at the same time~$H$ is a subgroup of $\PGL_2(L)$, then by Corollary~\ref{sled:J(PGL_2)} we immediately get~$J(H)\leqslant 60$. Suppose $H$ is not a subgroup of $\PGL_2(L)$, then we have a short exact sequence $$1\xrightarrow{}{} H^0\xrightarrow{}{} H\xrightarrow{}{}\mathbb{Z}/2\mathbb{Z}\xrightarrow{}{} 0,$$ where $H^0$ is the intersection of $H$ with $(\PGL_2(L),0)$. 
If $H^0$ is not isomorphic to $\mathfrak{A}_5$, then by Remark~\ref{rem:D2n,A4,S4,A5} we have $J(H^0)\leqslant 6$, and $J(H)\leqslant 12$. If $H^0$ is isomorphic to $\mathfrak{A}_5$, then according to Proposition~\ref{prop: A_n x Z2 or S_n}, group $H$ is isomorphic to either $\mathfrak{A}_5\times\mathbb{Z}/2\mathbb{Z}$, or $\mathfrak{S}_5$. Let us show that $H$ cannot be isomorphic to $\mathfrak{S}_5$. 

Suppose that $H\simeq \mathfrak{S}_5$. Consider an element $(B, 0) \in H^0$ of order $5$, consider an element~$(R,1) \in H$ of order $2$ and conjugate one by other: $$(R,1)(B,0)(R,1) = (R\overline{B} \overline{R}, 0) \in H^0.$$ Since element $(R,1)$ has order $2$, then $\overline{R} = R^{-1}$ and we can rewrite  $$(R\overline{B} \overline{R}, 0) =  (R \overline{B} R^{-1}, 0) \in H^0.$$ According to Lemma~\ref{lemma: conjugacy classes in A_n}, conjugacy class of an element $(B,0)$ in $H$ splits into two conjugacy classes in $H^0$. It is easy to see that elements $(B,0)$ and $(B^2, 0)$ lies in different conjugacy classes in $H^0$. Also it is easy to see that elements $(B,0)$ and $(R \overline{B} R^{-1}, 0)$ lies in different conjugacy classes in $H^0$. Therefore $(B^2,0)$ and $(R \overline{B} R^{-1}, 0)$ lies in the same conjugacy class in $H^0$. In particular there exists an element $\tilde{R} \in \PGL_2(L)$ such that $$\tilde{R}B^2\tilde{R}^{-1} = R \overline{B} R^{-1}.$$ We obtain the following $$\frac{\tr^2(B^2)}{\det(B^2)} = \frac{\tr^2(\tilde{R}B^2\tilde{R}^{-1})}{\det(\tilde{R}B^2\tilde{R}^{-1})} = \frac{\tr^2(R\overline{B}R^{-1})}{\det(R\overline{B}R^{-1})} =  \frac{\tr^2(\overline{B})}{\det(\overline{B})} = \overline{\left(\frac{\tr^2(B)}{\det(B)}\right)}.$$ But $B$ is an element of $\PGL_2(L)$ of order $5$, then by Lemma~\ref{lemma: tr/det} we have $$\frac{\tr^2(B)}{\det(B)} \in \mathbb{Q}(\sqrt{5}) \subset K.$$ Therefore Galois involution of extension $K \subset L$ acts on the scalar $\frac{\tr^2(B)}{\det(B)}$ trivially, and we get an equality $$\frac{\tr^2(B^2)}{\det(B^2)} = \overline{\left(\frac{\tr^2(B)}{\det(B)}\right)} = \frac{\tr^2(B)}{\det(B)}.$$ But this is a contradiction with Lemma~\ref{lemma: tr/det}. So $H$ cannot be isomorphic to $\mathfrak{S}_5$. 

Therefore, $H \simeq \mathfrak{A}_5 \times \mathbb{Z}/2\mathbb{Z}$, and $J(H) = 60$.

Thus, all possible cases of finite subgroups are considered and the lemma is proved.
\end{proof}

\section{Computation of \boldmath$M(K)$}\label{sect: main theorem}

Consider a field $K$. As mentioned in the introduction, we want to compute the following value:$$ M(K) = \max\limits_{X}(J(\Aut(X)),$$ where the maximum is taken over all smooth rational quadrics in $\mathbb{P}^3_K$.

The following proposition is standard. It will be needed in order to use general results about del Pezzo surfaces of degree 8.

\begin{propos}\label{prop:dp8 --> quadric}
    Let $X$ be a smooth rational surface over a field $K$. The following conditions are equivalent.
    \begin{enumerate}
    \item $X$ is isomorphic to a smooth quadric in $\mathbb{P}^3_K$;
    \item $X$ is a del Pezzo surface of degree $8$ such that $X_{\overline{K}} \simeq \mathbb{P}_{\overline{K}}^1 \times \mathbb{P}_{\overline{K}}^1$.
    \end{enumerate}
\end{propos}

\begin{proof}
Condition 1 obviously implies 2. Let us prove the implication in the other direction. Consider the exact sequence of groups \hbox{(see~\cite[Exercise 3.3.5(iii)]{G-S})}: $$0 \xrightarrow{} \text{Pic}(X) \xrightarrow{} \text{Pic}(X_{\overline{K}})^{\Gal(\overline{K}/K)} \xrightarrow{} \text{Br}(X) \xrightarrow{} \text{Br}(K(X)).$$ Since $X$ is rational, there is a $K$-point on $X$ by the Lang--Nishimura theorem (see for example~\hbox{\cite[Theorem 3.6.11]{Poo}}). Therefore the last homomorphism is an embedding and $$\text{Pic}(X) \simeq \text{Pic}(X_{\overline{K}})^{\Gal(\overline{K}/K)}.$$ This means that the class of divisors of the bidegree $(1,1)$ is defined over $K$, and it defines an embedding of $X$ in $\mathbb{P}^3_K$ as a smooth quadric.\end{proof}

There is an explicit description of del Pezzo surfaces of degree 8, which become isomorphic to the product of two projective lines when passing to the algebraic closure. Let~$K\subset L$ be a finite extension of fields, and $Y$ be a variety over $L$. By $R_{L/K}(Y)$, we denote the  Weil restriction of scalars (see, for example,~\cite[\S 8]{G-S}).

\begin{lemma}[{\cite[Lemma 7.3]{SV}}]\label{lemma:C1xC2 or R}
\begin{enumerate}

\item Let $X$ be a del Pezzo surface of degree $8$ over a field $K$ such that $X_{\overline{K}} \simeq \mathbb{P}_{\overline{K}}^1 \times \mathbb{P}_{\overline{K}}^1$. Then either $\rk\Pic(X) = 2$ and $X$ is isomorphic to a product $C \times C'$ of two conics over $K$, or $\rk\Pic(X) = 1$ and $X$ is isomorphic to~$R_{L/K}(Q)$, where $L \supset K$ is a quadratic separable extension and $Q$ is a conic over $L$.

\item Let $C$ be a smooth conic over $K$, and $L \supset K$ be a quadratic separable extension. Then $$\Aut(R_{L/K}(C_L)) \simeq \Aut(C_L) \rtimes \mathbb{Z}/2\mathbb{Z},$$ where the nontrivial element of the group $\mathbb{Z}/2\mathbb{Z}$ acts by the Galois involution of the extension~$L \supset K$.
\end{enumerate}
\end{lemma}

From Proposition~\ref{prop:dp8 --> quadric} and Lemma~\ref{lemma:C1xC2 or R} we get the following corollary.

\begin{sled}\label{sled: P1 x P1 or R}
Let $K$ be a field of characteristic $0$.
\begin{enumerate}
    \item Let $X$ be a smooth rational quadric in $\mathbb{P}^3_K$. Then either $\rk\Pic(X) = 2$ and~$X$ is isomorphic to $\mathbb{P}^1_K\times\mathbb{P}^1_K$, or $\rk\Pic(X) =1$ and $X$ is isomorphic to $R_{L/K}(\mathbb{P}^1_L)$, where~$L \supset K$ is a quadratic extension.
    \item Let $L \supset K$ be a quadratic extension, then $$\Aut(R_{L/K}(\mathbb{P}^1_L))\simeq\PGL_2(L)\rtimes\mathbb{Z}/2\mathbb{Z},$$ where the nontrivial element of the group $\mathbb{Z}/2\mathbb{Z}$ acts by the Galois involution of the extension~$L \supset K$.
    \end{enumerate}
 \end{sled}

\begin{proof}
We should immediately note that the rationality of $X$ implies the existence of a $K$-point on $X$ according to the Leng--Nishimura theorem (see for example~\mbox{\cite[Theorem 3.6.11]{Poo}}).

Applying Proposition~\ref{prop:dp8 --> quadric}, we conclude that $X$ is a del Pezzo surface of degree 8 such that~$X_{\overline{K}} \simeq \mathbb{P}_{\overline{K}}^1 \times \mathbb{P}_{\overline{K}}^1$. Then, by Lemma~\ref{lemma:C1xC2 or R}, $X$ is isomorphic to either a product~$C \times C'$ of two conics over $K$ or Weil scalar restriction $R_{L/K}(Q)$, where $K \subset L$ is a quadratic extension and $Q$ is a conic over $L$. Surface $C\times C'$ contains $K$-point if and only if $C\simeq C'\simeq\mathbb{P}^1_K$, and~$R_{L/K}(Q)$ contains $K$-point if and only if $Q$ contains $L$-point \hbox{(see~\cite[Exercise 8.1.2(iv)]{G-S})}, that is, $Q \simeq \mathbb{P}^1_L$. Thus, assertion 1 is proved.

Assertion 2 follows immediately from assertion 2 of Lemma~\ref{lemma:C1xC2 or R} and existence of isomorphism~$\Aut(\mathbb{P}^1_L)\simeq\PGL_2(L)$.\end{proof}

Now we can rewrite value $M(K)$ as follows: $$M(K) = \max\big(J(\Aut(\mathbb{P}_K^1 \times\mathbb{P}_K^1)), \max\limits_{[L:K] = 2}(J(\PGL_2(L) \rtimes \mathbb{Z}/2\mathbb{Z}))\big)$$ and use the results obtained in Section~\ref{sect:P1xP1} and Section~\ref{sect: semidirect 2}.




In the following propositions, the value $M(K)$ is computed depending on the conditions on a field $K$ introduced in Theorem~\ref{theo:M(K)}.

\begin{propos}\label{prop: M(K) = 7200, (i), (ii)}
Let $K$ be a field of characteristic $0$, such that $\sqrt{5}\in K$, and $-1$ is a sum of two squares in $K$. Then $M(K) = 7200$.
\end{propos}

\begin{proof}
Since $\sqrt{5}\in K$, and $-1$ is a sum of two squares in $K$, then by Theorem~\ref{theo:theo P^1 x P^1} $$J(\Aut(\mathbb{P}_K^1 \times \mathbb{P}_K^1)) = 7200.$$
Moreover, for any quadratic extension $L \supset K$, by Lemma~\ref{lemma: |G| < 120}, we have the inequality $$J(\PGL_2(L)\rtimes\mathbb{Z}/2\mathbb{Z}) \leqslant 120.$$ Therefore $M(K) = 7200$.\end{proof}

\begin{propos}\label{prop: M(K) = 120, not (i), but (iii)}
Let $K$ be a field of characteristic $0$ such that $\sqrt{5}\not\in K$, and $-1$ is a sum of two squares in $K(\sqrt{5})$. Then $M(K) = 120$.
\end{propos}

\begin{proof}
Since $\sqrt{5}\not\in K$, then by the Theorem~\ref{theo:theo P^1 x P^1} $$J(\Aut(\mathbb{P}_K^1 \times \mathbb{P}_K^1)) \leqslant 72.$$ Also for any quadratic extension $L\supset K$, in accordance with Lemma~\ref{lemma: |G| < 120}, we have the inequality $$J(\PGL_2(L)\rtimes\mathbb{Z}/2\mathbb{Z}) \leqslant 120.$$ Therefore, it is sufficient to present a quadratic extension $L_0\supset K$ such that $$J(\PGL_2(L_0)\rtimes\mathbb{Z}/2\mathbb{Z}) = 120.$$ By Lemma~\ref{lemma: S_5 action}, we can take $L_0$ equals to $K(\sqrt{5})$. Thus, the proposition is proved.\end{proof}

\begin{propos}\label{prop: M(K) = 60, (i), not (ii)} 
Let $K$ be a field of characteristic $0$ such that $\sqrt{5}\in K$, and~$-1$ is not a sum of two squares in $K$. Then $M(K) = 60$.
\end{propos}

\begin{proof}
Since $-1$ is not a sum of two squares in $K$, then by Theorem~\ref{theo:theo P^1 x P^1}, $$J(\Aut(\mathbb{P}_K^1\times\mathbb{P}_K^1)) = 8. $$ Also for any quadratic extension $L\supset K$, according to Lemma~\ref{lemma: J = 60 for every quadratic extension}, we have the inequality $$J(\PGL_2(L)\rtimes\mathbb{Z}/2\mathbb{Z})\leqslant 60.$$ At the same time, for any quadratic extension $L'\supset K$ such that $-1$ is a sum of two squares in $L'$, according to Proposition~\ref{prop:Beauville}, group $\PGL_2(L')$ contains a group isomorphic to~$\mathfrak{A}_5$. Therefore, we have $$J(\PGL_2(L') \rtimes \mathbb{Z}/2\mathbb{Z})\geqslant J(\PGL_2(L')) \geqslant J(\mathfrak{A}_5) = 60.$$
Thus, $M(K) = 60$, and the proposition is proved.
\end{proof}

\begin{propos}\label{prop: M(K) = 8, not (i), not (iii)} 
Let $K$ be a field of characteristic $0$ such that $\sqrt{5}\not \in K$, and $-1$ is not a sum of two squares in $K(\sqrt{5})$. Then $M(K) = 8$.
\end{propos}

\begin{proof}
Since $-1$ is not a sum of two squares in $K(\sqrt{5})$, then, in particular, $-1$ is not the sum of two squares in $K$, so by Theorem~\ref{theo:theo P^1 x P^1}, $$J(\Aut(\mathbb{P}_K^1\times\mathbb{P}_K^1)) = 8.$$ Also, since $\sqrt{5}\not\in K$, and $-1$ is not a sum of two squares in $K(\sqrt{5})$, then from Proposition~\ref{prop:Beauville} it follows that for any quadratic extension $L\supset K$, the group~$\PGL_2(L)$ does not contain a subgroup isomorphic to $\mathfrak{A}_5$. Applying Lemma~\ref{lemma: sqrt(5) - a,b -}, we obtain the inequality for any quadratic field extension~$L\supset K$: $$J(\PGL_2(L)\rtimes\mathbb{Z}/2\mathbb{Z}) \leqslant 6.$$ Therefore, $M(K) = 8$. \end{proof}

Now everything is prepared for proving the main results.

\begin{proof}[Proof of Theorem~\ref{theo:M(K)}.] 
If $\sqrt{5}\in K$, and $-1$ is the a of two squares in $K$, then~$M(K)=7200$ according to Proposition~\ref{prop: M(K) = 7200, (i), (ii)}. If $\sqrt{5}\in K$, but $-1$ is not a sum of two squares in $K$, then $M(K) = 60$ according to Proposition~\ref{prop: M(K) = 60, (i), not (ii)}. If $\sqrt{5}\not\in K$, and~$-1$ is a sum of two squares in $K(\sqrt{5})$, then $M(K) = 120$ according to Proposition~\ref{prop: M(K) = 120, not (i), but (iii)}. And finally, if $\sqrt{5}\not\in K$, and $-1$ is not a sum of two squares in $K(\sqrt{5})$, then~$M(K)= 8$ according to Proposition~\ref{prop: M(K) = 8, not (i), not (iii)}. Thus, we have considered all possible fields of characteristic zero, and the theorem is proved. \end{proof}

\begin{proof}[Proof of Proposition~\ref{prop: J(AutS) <= 120 if rkPic = 1}] From Corollary~\ref{sled: P1 x P1 or R}, we have the isomorphism $$S\simeq R_{L/K}(\mathbb{P}^1_L),$$ where~$K\subset L$ is a quadratic extension, and $$\Aut(S) \simeq \PGL_2(L) \rtimes \mathbb{Z}/2\mathbb{Z}.$$ By Lemma~\ref{lemma: |G| < 120}, we get an estimate $J(\Aut(S)) \leqslant 120$. \end{proof}

\begin{proof}[Proof of Corollary~\ref{sled: M(Q), M(R), M(C)...}.] Theorem~\ref{theo:M(K)} obviously implies that $$M(\mathbb{Q}) = 8, \; M(\mathbb{R}) = 60, \; M(\mathbb{C}) = 7200. $$

Consider the field $\mathbb{Q}(\sqrt{-7})$. Firstly, note that $\sqrt{5}\not \in\mathbb{Q}(\sqrt{-7})$. Secondly, note that $-1$ is a sum of two squares in the field $\mathbb{Q}(\sqrt{-7},\sqrt{5})$, for example: $$\Bigl(\frac{\sqrt{-7} +\sqrt{5}}{-1 + \sqrt{-35}}\Bigl)^2 + \Bigl(\frac{6}{-1 + \sqrt{-35}}\Bigl)^2 = -1.$$
Thus, by Theorem~\ref{theo:M(K)}, we have $M(\mathbb{Q}(\sqrt{-7})) = 120$. 

Consider the field $\mathbb{Q}(i)$. Note that $\sqrt{5}$ does not lie in $\mathbb{Q}(i)$ and that $-1$ is a sum of two squares in $\mathbb{Q}(i)$, for example~$-1 = i^2+0^2$. Therefore, by Theorem~\ref{theo:M(K)}, we have~\mbox{$M(\mathbb{Q}(i)) = 120$}.
\end{proof}

\end{document}